\documentclass[12pt]{article}
\usepackage{amssymb,amsmath,mathrsfs,amsthm}
\usepackage{hyperref}
\usepackage{graphicx}
\usepackage{color}
\usepackage[OT2,OT1]{fontenc}
\usepackage{upquote}
\usepackage{authblk}
\usepackage[numbers,sort&compress]{natbib}

\newtheorem{theorem}{Theorem}[section]
\newtheorem{lemma}[theorem]{Lemma}

\usepackage{xcolor}
\usepackage{framed}
\definecolor{shadecolor}{cmyk}{0,0,0,0.05} 

\usepackage{alltt} 

\begin{document}

\title{\LARGE\bf Algorithmic determination \\
of a large integer in the two-term \\
Machin-like formula for $\pi$}

\author[1, 2, 3]{\small Sanjar M. Abrarov}
\author[2, 3, 4]{\small Rehan Siddiqui}
\author[3, 4]{\small Rajinder K. Jagpal}
\author[1, 2, 4]{\small \\ Brendan M. Quine}

\affil[1]{\scriptsize Thoth Technology Inc., Algonquin Radio Observatory, Achray Rd, RR6, Pembroke, Canada, K8A 6W7 \normalsize}
\affil[2]{\scriptsize Dept. Earth and Space Science and Engineering, York University, 4700 Keele St., Canada, M3J 1P3 \normalsize}
\affil[3]{\scriptsize Epic College of Technology, 5670 McAdam Rd., Mississauga, Canada, L4Z 1T2 \normalsize}
\affil[4]{\scriptsize Dept. Physics and Astronomy, York University, 4700 Keele St., Toronto, Canada, M3J 1P3 \normalsize}

\date{September 12, 2021}
\maketitle

\vspace{-0.5cm}
\begin{abstract}
In our earlier publication we have shown how to compute by iteration a rational number $u_{2,k}$ in the two-term Machin-like formula for $\pi$ of the kind $$\frac{\pi}{4}=2^{k-1}\arctan\left(\frac{1}{u_{1,k}}\right)+\arctan\left(\frac{1}{u_{2,k}}\right),\qquad k\in\mathbb{Z},\quad k\ge 1,$$ where $u_{1,k}$ can be chosen as an integer $u_{1,k}=\left\lfloor{a_k/\sqrt{2-a_{k - 1}}}\right\rfloor$ with nested radicals defined as $a_k=\sqrt{2+a_{k - 1}}$ and $a_0 = 0$. In this work, we report an alternative method for determination of the integer $u_{1,k}$. This approach is based on a simple iteration and does not require any irrational (surd) numbers from the set $\left\{a_k\right\}$ in computation of the integer $u_{1,k}$. Mathematica programs validating these results are presented.
\vspace{0.25cm}
\vspace{0.25cm}
\end{abstract}

\section{Introduction}

Historically, a computation of decimal digits of $\pi$ was a big challenge until 1706, when the English astronomer and mathematician John Machin discovered a two-term formula for $\pi$ as given by
\begin{equation}\label{eq_1}
\frac{\pi }{4} = 4\arctan \left( {\frac{1}{5}} \right) - \arctan \left( {\frac{1}{{239}}} \right),
\end{equation}
that is named in his honor now. Using this remarkable formula he first was able to calculate $100$ decimal digits of $\pi$ \cite{Beckmann1971, Berggren2004, Borwein2008}. Nowadays the identities of kind
\begin{equation}\label{eq_2}
\frac{\pi }{4} = \sum\limits_{j = 1}^J {{A_j}\arctan \left( {\frac{1}{{{B_j}}}} \right)},
\end{equation}
where ${A_j}$ and ${B_j}$ are rational numbers, are regarded as the Machin-like formulas for $\pi$. Interestingly that some of them, including the original Equation~\eqref{eq_1}, can be proved geometrically~\cite{Nelsen1990, Popesku2021}. It is very often when in the Machin-like formulas for $\pi$ the constants ${A_j}$ and ${B_j}$ are both integers~\cite{Lehmer1938, Wetherfield1996, Chien-Lih1997}. The more complete lists of formulas of kind \eqref{eq_2} can be found in the references~\cite{URL1,URL2} and weblinks provided~therein.

The significance of the Machin-like formulas cannot be overestimated as their application may be one of the most efficient ways in computing $\pi$. Historically, only these formulas were able to compete with Chudnovsky formula \cite{Berggren2004,Ravi2013} to beat the records in computation of the decimal digits of $\pi$. In particular, in~2002, Kanada first computed more than one trillion digits of $\pi$ by using the following self-checking pair of the Machin-like formulas~\cite{Calcut2009}
\small
$$
\frac{\pi }{4} = 44\arctan \left( {\frac{1}{{57}}} \right) + 7\arctan \left( {\frac{1}{{239}}} \right) - 12\arctan \left( {\frac{1}{{682}}} \right) + 24\arctan \left( {\frac{1}{{12943}}} \right)
$$
\normalsize
and
\small
$$
\frac{\pi }{4} = 12\arctan \left( {\frac{1}{{49}}} \right) + 32\arctan\left( {\frac{1}{{57}}} \right) - 5\arctan \left( {\frac{1}{{239}}} \right) + 12\arctan \left( {\frac{1}{{110443}}} \right).
$$
\normalsize
Such an achievement made by Kanada shows a colossal potential of the Machin-like formulas for computation of the decimal digits of $\pi$.

As a simplest case one can apply the Maclaurin expansion series for computation of the arctangent functions in Equation~\eqref{eq_2}
$$
\arctan \left( x \right) = x - \frac{{{x^3}}}{3} + \frac{{{x^5}}}{5} - \frac{{{x^7}}}{7} +  \cdots  = \sum\limits_{n = 0}^\infty  {\frac{{{{\left( { - 1} \right)}^n}{x^{2n + 1}}}}{{2n + 1}}}
$$
and since this equation implies
\begin{equation}\label{eq_3}
\arctan \left( x \right) = x + O\left( {{x^3}} \right),
\end{equation}
we can conclude that it would be very desirable to have the coefficient ${B_j}$ as large as possible by absolute value in order to improve the convergence~rate.

Although the Maclaurin expansion series of the arctangent function can be simply implemented, its application is not optimal. The~more efficient way to compute the arctangent functions in Equation~\eqref{eq_2} is to use the Euler's expansion formula~\cite{Chien-Lih2005}
$$
\arctan \left( x \right) = \sum\limits_{n = 0}^\infty  {\frac{{{2^{2n}}{{\left( {n!} \right)}^2}}}{{\left( {2n + 1} \right)!}}\frac{{{x^{2n + 1}}}}{{{{\left( {1 + {x^2}} \right)}^{n + 1}}}}} .
$$
Alternatively, the following expansion series
$$
\arctan \left( x \right) = 2\sum\limits_{n = 1}^\infty  {\frac{1}{{2n - 1}}\frac{{{g _n}\left( x \right)}}{{g _n^2\left( x \right) + h _n^2\left( x \right)}}} ,
$$
where
$$
{g _1}\left( x \right) = 2/x, \; {h _1}\left( x \right) = 1,
$$ 
$$
{g _n}\left( x \right) = \left( {1 - 4/{x^2}} \right){g _{n - 1}} + 4{h _{n - 1}}\left( x \right)/x,
$$
$$
{h _n}\left( x \right) = \left( {1 - 4/{x^2}} \right){h _{n - 1}} - 4{g _{n - 1}}\left( x \right)/x,
$$
can also be used for more rapid convergence. This formula 
 can be obtained by a trivial rearrangement of the Equation~(5) from our work~\cite{Abrarov2018a} (see also~\cite{Abrarov2017a}).

In 1938 Lehmer introduced a measure defined as~\cite{Lehmer1938, Tweddle1991}
$$
\mu  = \sum\limits_{j = 1}^J {\frac{1}{{{{\log }_{10}}\left( {\left| {{B_j}} \right|} \right)}}}.
$$
This measure can be used to determine a computational efficiency of a given Machin-like formula for $\pi$. Specifically, when value of the constant $\mu $ is smaller, then less computational labour is required to compute $\pi$ by a given Machin-like formula. Therefore, it is very desirable to reduce the number of the terms $J$ and to increase ${B_j}$ by absolute value. The~more detailed information about the Lehmer’s measure $\mu$ and its significance for efficient computation of $\pi$ can be found in literature~\cite{Wetherfield1996}.

In our previous publication using de Moivre's formula we derived the two-term Machin-like formula for $\pi$ \cite{Abrarov2017a}
\small
\begin{equation}\label{eq_4}
\frac{\pi}{4}=2^{k-1}\arctan\left(\frac{1}{u_{1,k}}\right)+\arctan\left(\frac{1-\sin\left(2^{k-1}\arctan\left(\frac{2u_{1,k}}{u_{1,k}^2-1}\right)\right)}{\cos\left(2^{k-1}\arctan\left(\frac{2u_{1,k}}{u_{1,k}^2-1}\right)\right)}\right),
\end{equation}
\normalsize
where ${u_{1,k}}$ can be chosen as an integer
\begin{equation}\label{eq_5}
{u_{1,k}} = \left\lfloor {\frac{{{a_k}}}{{\sqrt {2 - {a_{k - 1}}} }}} \right\rfloor,
\end{equation}
such that a set of nested radicals $\left\{ {{a_k}} \right\}$ can be computed as ${a_k} = \sqrt {2 + {a_{k - 1}}} $ starting from $a_0=0$.

In the recent publication the researcher(s) from the Wolfram Mathematica~\cite{WolframCloud} demonstrated an example for computation of $\pi$ as a rational fraction. In~particular, it was shown that with integer constant ${u_{1,1000}}$ the constant $\pi$ can be computed such that the ratio
$$
\frac{\pi }{4} \approx \frac{{{2^{999}}}}{{{u_{1,1000}}}}
$$
results in more than $300$ correct decimal digits. However, the~method of computation shown in~\cite{WolframCloud} is based on Equation~\eqref{eq_5} that involves the nested radicals consisting of multiple square roots of $2$ \cite{Herschfeld1935, Servi2003, Levin2005, Kreminski2008, Borwein1991, Rao2005}. Although~Equation~\eqref{eq_5} helps generate the required integers $u_{1,k}$, it should not be generally used at larger values $k$ since a function based on multiple square roots is not an elementary. Therefore, a~simple method based on rational approximation would be preferable. In~this work, we develop a new method of computation of the integer ${u_{1,k}}$ that excludes application of the set of nested radicals $\left\{ {{a_k}} \right\}$. This approach is simple and does not require any irrational (surd) numbers in~computation.

\section{Preliminaries}

Suppose that
\begin{equation}\label{eq_6}
\frac{\pi }{4} = \alpha \arctan \left( {\frac{1}{\gamma }} \right), \qquad \alpha, \, \gamma  \in \mathbb{R}.
\end{equation}
The simplest case when $\alpha  = \gamma  = 1$. However, there may be infinitely many identities of kind \eqref{eq_6}. Let us show the infinitude of this kind of~formulas.

\begin{theorem}\label{th_1}
There are infinitely many numbers $\alpha$ and $\gamma$ satisfying the relation \eqref{eq_6}.
\end{theorem}

\begin{proof}
The proof becomes straightforward by considering the following example
\begin{equation}\label{eq_7}
\frac{\pi }{4} = 2^{k - 1}\arctan\left(\frac{\sqrt{2 - a_{k - 1}}}{  a_k} \right), \qquad k \in \mathbb{Z}, \qquad k \ge 1,
\end{equation}
where nested radicals are computed as ${a_k} = \sqrt {2 + {a_{k - 1}}} $ at ${a_0} = 0$. Comparing \linebreak Equations \eqref{eq_6} with \eqref{eq_7} immediately yields that $\alpha = 2^{k-1}$ and $\gamma  = {a_k}/\sqrt {2 - {a_{k - 1}}}$.
The derivation of Equation~\eqref{eq_7} is very simple and can be found in~\cite{Abrarov2018a}.
\end{proof}

It is interesting to note that using Equation~\eqref{eq_7} one can easily prove the well-known Equation~\eqref{eq_8} for $\pi$ below. 

\begin{theorem}\label{th_2}
We have that
\begin{equation}\label{eq_8}
\pi  = \mathop {\lim }\limits_{k \to \infty } {2^k}\sqrt {2 - \underbrace{\sqrt {2 + \sqrt {2 + \sqrt {2 +  \cdots }}}}_{k - 1\,\rm{square\, roots}}}.
\end{equation}
\end{theorem}
	
\begin{proof}
The following relation
$$
\mathop {\lim }\limits_{k \to \infty } {a_k} =\sqrt {2 + \sqrt {2 + \sqrt {2 + \sqrt {2 \cdots } } } }
$$
is a simplest Ramanujan's nested radical~\cite{Herschfeld1935, Servi2003, Levin2005, Kreminski2008, Borwein1991, Rao2005}. Denote $X$ as unknown, then from the~relation
$$
\sqrt {2 + \sqrt {2 + \sqrt {2 + \sqrt {2 \cdots } } } }  = X
$$
it immediately follows that
$$
\sqrt {2 + X}  = X
$$
or
$$
2 + X = {X^2}.
$$

Solving this equation yields two solutions for $X$ that are $ - 1$ and $2$. {Since}
$$
0<\sqrt{2}<\sqrt{2+\sqrt{2}}<\sqrt{2+\sqrt{2+\sqrt{2}}}\dots
$$
{all values of ${a_k}$ are non-negative and monotonically increase with increasing integer $k$.} Therefore, excluding $ - 1$ from consideration we end up with a solution
$$
\mathop {\lim }\limits_{k \to \infty } {a_k} = 2.
$$
From this limit it immediately follows that
$$
\mathop {\lim }\limits_{k \to \infty } \sqrt {2 - {a_{k - 1}}}  = \mathop {\lim }\limits_{k \to \infty } \sqrt {2 - {a_k}}  = 0.
$$
Consequently, the ratio $\sqrt {2 - {a_{k - 1}}} /{a_k} \to 0$ as $k \to \infty$.

The argument of the arctangent function in Equation~\eqref{eq_7} tends to zero as the integer $k$ increases. {Therefore, in~accordance with relation} \eqref{eq_3} {we can write}
$$
\arctan \left( {\frac{{\sqrt {2 - {a_{k - 1}}} }}{{{a_k}}}} \right) \to 0, \quad \frac{{\sqrt {2 - {a_{k - 1}}} }}{{{a_k}}} \to 0 \quad \rm{at} \quad k \to \infty,
$$
from which it follows that
\small
$$
\frac{\pi }{4} = \mathop {\lim }\limits_{k \to \infty } {2^{k - 1}}\arctan \left( {\frac{{\sqrt {2 - {a_{k - 1}}} }}{{{a_k}}}} \right) =\mathop {\lim }\limits_{k \to \infty } {2^{k - 1}} {\frac{{\sqrt {2 - {a_{k - 1}}} }}{{{a_k}}}} = \mathop {\lim }\limits_{k \to \infty } {2^{k - 2}}\sqrt {2 - {a_{k - 1}}}
$$
\normalsize
or
$$
\pi  = \mathop {\lim }\limits_{k \to \infty } {2^k}\sqrt {2 - {a_{k - 1}}}.
$$
\end{proof}

As we can see, this proof of Equation~\eqref{eq_8} is as easy as the one shown in~\cite{Herschfeld1935}. We also note that the following limit that we used in the proof
$$
\frac{\pi }{4} = \mathop {\lim }\limits_{k \to \infty } {2^{k - 1}}\arctan \left( {\frac{{\sqrt {2 - {a_{k - 1}}} }}{{{a_k}}}} \right)
$$
is valid since the identity \eqref{eq_7} remains valid at any arbitrarily large positive integer $k$.

It is also easy to prove the infinitude of the Machin-like formulas \eqref{eq_2} for $\pi$; substituting $x=1$ into the expansion series~\cite{Abrarov2021}
\[
\begin{aligned}
\arctan\left(x\right)&=\sum_{n=1}^N\arctan\left(\frac{N x}{N^2+\left(n-1\right)n x^2}\right) \\
&\Leftrightarrow\arctan\left(N x\right)=\sum_{n=1}^N\arctan\left(\frac{x}{1+\left(n-1\right)n x^2}\right)
\end{aligned}
\]
we obtain the following identity~\cite{OEIS2021}
\begin{equation}\label{eq_9}
\frac{\pi }{4} = \sum\limits_{n = 1}^N {\arctan } \left( {\frac{N}{{\left( {n - 1} \right)n + {N^2}}}} \right)
\end{equation}
leading to
\[
\frac{\pi }{4} = \arctan \left( 1 \right), \qquad N = 1,
\]
\[
\frac{\pi }{4} = \arctan \left( {\frac{1}{3}} \right) + \arctan \left( {\frac{1}{2}} \right), \qquad N = 2,
\]
\[
\frac{\pi }{4} = \arctan \left( {\frac{1}{5}} \right) + \arctan \left( {\frac{3}{{11}}} \right) + \arctan \left( {\frac{1}{3}} \right), \qquad N = 3,
\]
\[
\frac{\pi }{4} = \arctan \left( {\frac{1}{7}} \right) + \arctan \left( {\frac{2}{{11}}} \right) + \arctan \left( {\frac{2}{9}} \right) + \arctan \left( {\frac{1}{4}} \right), \qquad N = 4
\]
and so on. Although~the identity \eqref{eq_9} shows infinitude of the Machin-like formulas for $\pi$, its number of the terms increases with increasing $N$. We can also show a simple proof for infinitude of the two-term Machin-like formulas for $\pi$.

\begin{lemma}\label{lm_1}
{For real $\alpha$ and $\beta_1$, there are infinitely many two-terms Machin-like formulas for $\pi$ of~kind}
\begin{equation}\label{eq_10}
\frac{\pi }{4} = \alpha \arctan \left( {\frac{1}{{{\beta _1}}}} \right) + \arctan \left( {\frac{1}{{{\beta _2}}}} \right).
\end{equation}
\end{lemma}

\begin{proof}
The Lemma \ref{lm_1} follows directly from the Theorem \ref{th_1} that implies infinitude of equations of kind \eqref{eq_6}. In~order to show this relation, we assume that $\alpha $ and $\gamma $ in Equation~\eqref{eq_6} are both positive numbers and represent  $\gamma$  as a sum $\gamma  = \zeta  + \delta$, where $\delta$ is any small number that can be chosen arbitrarily such that $\gamma  >> \left| \delta  \right|$. Thus, we can rewrite the Equation~\eqref{eq_6} in~form
$$
\frac{\pi }{4} = \alpha \arctan \left( {\frac{1}{{\zeta  + \delta }}} \right).
$$

From the inequality $\gamma  >> \left| \delta  \right|$ it follows that $\zeta  \approx \beta $. Therefore, we can approximate
$$
\arctan \left( {\frac{1}{\gamma }} \right) \approx \arctan \left( {\frac{1}{\zeta }} \right).
$$
By introducing now an error term $\varepsilon $, we can infer that
$$
\arctan \left( {\frac{1}{{\zeta  + \delta }}} \right) = \arctan \left( {\frac{1}{\zeta }} \right) + \varepsilon.
$$
Consequently, we have
$$
\frac{\pi }{4} = \alpha \arctan \left( {\frac{1}{\zeta }} \right) + \varepsilon
$$
or
\begin{equation}\label{eq_11}
\frac{\pi }{4} = \alpha \arctan \left( {\frac{1}{\zeta }} \right) + \arctan \left( {\frac{1}{\eta }} \right),
\end{equation}
where $\eta $ is defined such that
$$
\arctan \left( {\frac{1}{\eta }} \right) = \varepsilon.
$$

The Equation~\eqref{eq_11} is of the same kind as that of given by Equation~\eqref{eq_10}. This completes the proof since for any equation of kind \eqref{eq_6} we can always construct an equation of kind~\eqref{eq_10}.
\end{proof}

When $\alpha $ and ${\beta _1}$ in Equation~\eqref{eq_10} are known, then the unknown value ${\beta _2}$  is given by
\begin{equation}\label{eq_12}
{\beta _2} = \frac{2}{{{{\left( {\left( {{\beta _1} + i} \right)/\left( {{\beta _1} - i} \right)} \right)}^\alpha } - i}} - i.
\end{equation}
The derivation of Equation~\eqref{eq_12} can be shown from the following identity
$$
\arctan \left( {\frac{1}{x}} \right) = \frac{1}{{2i}}\ln \left( {\frac{{x + i}}{{x - i}}} \right)
$$
Thus, substituting this identity into Equation~\eqref{eq_10} results in
$$
\frac{\pi }{4} = \frac{\alpha }{{2i}}\ln \left( {\frac{{{\beta _1} + i}}{{{\beta _1} - i}}} \right) + \frac{1}{{2i}}\ln \left( {\frac{{{\beta _2} + i}}{{{\beta _2} - i}}} \right)
$$
or
$$
\frac{\pi }{2}i = \ln \left( {{{\left( {\frac{{{\beta _1} + i}}{{{\beta _1} - i}}} \right)}^\alpha }\frac{{{\beta _2} + i}}{{{\beta _2} - i}}} \right).
$$
Exponentiation on both sides leads to
\begin{equation}\label{eq_13}
{\left( {\frac{{{\beta _1} + i}}{{{\beta _1} - i}}} \right)^\alpha }\frac{{{\beta _2} + i}}{{{\beta _2} - i}} = i.
\end{equation}
{Solving this with respect to the constant $\beta_2$ leads to Equation} \eqref{eq_12}.

\begin{theorem}\label{th_3}
If in Equation~\eqref{eq_10} $\forall k \ge 2$ the multiplier $\alpha  = {2^{k - 1}}$ and ${\beta _1}$ is a rational number greater than $1$, then ${\beta _2}$ is also a rational number.
\end{theorem}

\begin{proof}
Define ${\sigma_1}$ and ${\tau_1}$ such that
$$
{\sigma_1} = {\mathop{\rm Re}\nolimits} \left[ {\frac{{{\beta _1} + i}}{{{\beta _1} - i}}} \right] = \frac{{\beta _{_1}^2 - 1}}{{\beta _{_1}^2 + 1}}
$$
and
$$
{\tau_1} = {\mathop{\rm Im}\nolimits} \left[ {\frac{{{\beta _1} + i}}{{{\beta _1} - i}}} \right] = \frac{{2{\beta _1}}}{{\beta _{_1}^2 + 1}}.
$$

Then, it is not difficult to see by induction that
\[
\begin{aligned}
{\left( {{\sigma_1} + i{\tau_1}} \right)^{2k - 1}} &= \overbrace {{{\left( {{{\left( {{{\left( {{{\left( {{\sigma_1} + i{\tau_1}} \right)}^2}} \right)}^2}} \right)}^{2 \cdots }}} \right)}^2}}^{k - 1\,{\rm{powers}}\,\,{\rm{of}}\,2} = \overbrace {{{\left( {{{\left( {{{\left( {{{\left( {{\sigma_2} + i{\tau_2}} \right)}^2}} \right)}^2}} \right)}^{2 \cdots }}} \right)}^2}}^{k - 2\,{\rm{powers}}\,\,{\rm{of}}\,2}\\
&= \overbrace {{{\left( {{{\left( {{{\left( {{{\left( {{\sigma_3} + i{\tau_3}} \right)}^2}} \right)}^2}} \right)}^{2 \cdots }}} \right)}^2}}^{k - 3\,{\rm{powers}}\,\,{\rm{of}}\,2} =  \cdots  = \overbrace {{{\left( {{{\left( {{{\left( {{{\left( {{\sigma_{n}} + i{\tau_{n}}} \right)}^2}} \right)}^2}} \right)}^{2 \cdots }}} \right)}^2}}^{ k-n\,{\rm{powers}}\,\,{\rm{of}}\,2} =  \cdots \\
&= {\left( {{{\left( {{\sigma_{k - 2}} + i{\tau_{k - 2}}} \right)}^2}} \right)^2} = {\left( {{\sigma_{k - 1}} + i{\tau_{k - 1}}} \right)^2} = {\sigma_k} + i{\tau_k},
\end{aligned}
\]
where by the following two-step iteration we have
\begin{equation}\label{eq_14}
\left\{ \begin{aligned}
{\sigma_n} &= \sigma_{n - 1}^2 - \tau_{n - 1}^2\\
{\tau_n} &= 2{\sigma_{n - 1}}{\tau_{n - 1}}, \qquad n = \left\{ {2,3,4, \ldots ,k} \right\}.
\end{aligned} \right.
\end{equation}
Consequently, Equation \eqref{eq_12} can be rewritten in form
\begin{equation}\label{eq_15}
{\beta_2} = \frac{2}{{{\sigma_k} + i{\tau_k} - i}} - i = \frac{{2{\sigma_k}}}{{\sigma_k^2 + {{\left( {{\tau_k} - 1} \right)}^2}}} + i\left( {\frac{{2\left( {1 - {\tau_k}} \right)}}{{\sigma_k^2 + {{\left( {{\tau_k} - 1} \right)}^2}}} - 1} \right).
\end{equation}

Applying the de Moivre's formula we can separate the complex number ${\left( {{\sigma_1} + i{\tau_1}} \right)^{2k - 1}}$ into real and imaginary parts in polar form as
\footnotesize
\[
{\left( {{\sigma_1} + i{\tau_1}} \right)^{2k - 1}} = {\left( {\sigma_1^2 + \tau_1^2} \right)^{2k - 2}}\left( {\cos \left( {{2^{k - 1}}{\rm{Arg}}\left( {{\sigma_1} + i{\tau_1}} \right)} \right) + i\sin \left( {{2^{k - 1}}{\rm{Arg}}\left( {{\sigma_1} + i{\tau_1}} \right)} \right)} \right).
\]
\normalsize
Substituting this expression into the Equation~\eqref{eq_15} after some trivial rearrangement we get
\[
{\beta _2} = \frac{{\cos \left( {{2^{k - 1}}{\rm{Arg}}\left( {\frac{{{\beta _1} + i}}{{{\beta _1} - i}}} \right)} \right)}}{{1 - \sin \left( {{2^{k - 1}}{\rm{Arg}}\left( {\frac{{{\beta _1} + i}}{{{\beta _1} - i}}} \right)} \right)}}.
\]

Since ${\beta_1} > 1$, then
\[
{\mathop{\rm Re}\nolimits} \left[ {\frac{{{\beta _1} + i}}{{{\beta _1} - i}}} \right] = \frac{{\beta _1^2 - 1}}{{\beta _1^2 + 1}} > 1
\]
and, therefore, the~principal value argument can be replaced by the arctangent function as~follows
\[
\begin{aligned}
{\rm{Arg}}\left( {\frac{{{\beta _1} + i}}{{{\beta _1} - i}}} \right) &= {\rm{Arg}}\left( {\frac{{\beta _1^2 - 1}}{{\beta _1^2 + 1}} + i\frac{{2{\beta _1}}}{{\beta _1^2 + 1}}} \right)\\
&= \arctan \left( {\left( {\frac{{2{\beta _1}}}{{\beta _1^2 + 1}}} \right)/\left( {\frac{{\beta _1^2 - 1}}{{\beta _1^2 + 1}}} \right)} \right) = \arctan \left( {\frac{{2{\beta _1}}}{{\beta _1^2 - 1}}} \right).
\end{aligned}
\]
Consequently, we can write
\begin{equation}\label{eq_16}
{\beta _2} = \frac{{\cos \left( {{2^{k - 1}}\arctan \left( {\frac{{2{\beta _1}}}{{\beta _1^2 - 1}}} \right)} \right)}}{{1 - \sin \left( {{2^{k - 1}}\arctan \left( {\frac{{2{\beta _1}}}{{\beta _1^2 - 1}}} \right)} \right)}}.
\end{equation}

As we can see from this equation, the constant $\beta_2 \in \mathbb{R}$. This signifies that the imaginary part of Equation~\eqref{eq_15} must be equal to zero. Therefore, from~Equation~\eqref{eq_15} we get
\[
\frac{{2\left( {1 - {\tau_k}} \right)}}{{\sigma_k^2 + {{\left( {{\tau_k} - 1} \right)}^2}}} - 1 = 0 \Leftrightarrow {\beta _2} = \frac{{{\sigma_k}}}{{1 - {\tau_k}}}.
\]
The values ${\sigma_k}$ and ${\tau_k}$ are rational since, according to two-step iteration \eqref{eq_14} all values ${\sigma_n}$ and ${\tau_n}$ at any intermediate steps of iterations are rational. Therefore, the~constant ${\beta _2}$ must be a rational number.
\end{proof}

Consider two examples. Choosing  $\alpha  = 16$ and ${\beta _1} = 509/25 = 20.36$ and substituting these two values into Equation~\eqref{eq_12} we can find that
$$
{\beta _2} = \frac{{114322283895863787286174872158832679853761}}{{19955894848381168459034791030978450561}}
$$

The following Mathematica code:
\vspace{-0.25cm}
\begin{shaded}
\begin{verbatim}
\[Alpha]=16;
\[Beta]1=509/25;
\[Beta]2=114322283895863787286174872158832679853761/
19955894848381168459034791030978450561;

Pi/4==16*ArcTan[1/\[Beta]1]+ArcTan[1/\[Beta]2]
\end{verbatim}
\end{shaded}
\vspace{-0.25cm}
\noindent returns \texttt{True}. Choosing now, for~example, $\alpha  = 16$ and ${\beta _1} = 407/20 = 20.35$ and substituting these two values into Equation~\eqref{eq_12} again, we can get a negative value
$$
{\beta _2} =  - \frac{{817344423776293722798294452010774302554561}}{{172199208235943812365929049219262848959}}.
$$
The following Mathematica code:
\vspace{-0.25cm}
\begin{shaded}
\begin{verbatim}
\[Alpha]=16;
\[Beta]1=407/20;
\[Beta]2=-817344423776293722798294452010774302554561/
172199208235943812365929049219262848959;

Pi/4==16*ArcTan[1/\[Beta]1]+ArcTan[1/\[Beta]2]
\end{verbatim}
\end{shaded}
\vspace{-0.25cm}
\noindent also validates the two-term Machin-like formula for $\pi$ by returning \texttt{True}. The~different signs in ${\beta _2}$ follow from the chosen integer $\alpha = 16 = {2^{k - 1}}$ at $k = 5$. When we take
\[
\begin{aligned}
{\beta _1} = \frac{{{a_5}}}{{\sqrt {2 - {a_4}} }} &= \frac{{\sqrt {2 + \sqrt {2 + \sqrt {2 + \sqrt {2 + \sqrt 2 } } } } }}{{\sqrt {2 - \sqrt {2 + \sqrt {2 + \sqrt {2 + \sqrt 2 } } } } }} \\
&= 20.355467624987188 \ldots \quad \left( {{\rm{irrational}}} \right),
\end{aligned}
\]
then the value of  ${\beta _2} = 0$. However, if~${\beta _1} < {a_k}/\sqrt {2 - {a_{k - 1}}} $, then ${\beta _2} < 0$ and vice~versa if ${\beta _1} > {a_k}/\sqrt {2 - {a_{k - 1}}} $, then ${\beta _2} > 0$. The~examples above correspond to the \linebreak following~inequality
$$
\underbrace {\frac{{407}}{{20}}}_{{\rm{case}}\,{\beta _2} < 0} < \underbrace {\frac{{\sqrt {2 + \sqrt {2 + \sqrt {2 + \sqrt {2 + \sqrt 2 } } } } }}{{\sqrt {2 - \sqrt {2 + \sqrt {2 + \sqrt {2 + \sqrt 2 } } } } }}}_{{\rm{case}}\,{\beta _2} = 0} < \underbrace {\frac{{509}}{{25}}}_{{\rm{case}}\,{\beta _2} > 0}.
$$

Chien-Lih proposed a method showing how to reduce the Lehmer's measure by using the Euler's-type identity in an iteration for generating the two-term Machin-like formulas for $\pi$. However, our method of generating the two-term Machin-like formula for $\pi$ based on the two-step iteration \eqref{eq_14} is much easier than the method proposed by Chien-Lih in the work~\cite{Chien-Lih2004}.

\section{Derivation}

Using Equation~\eqref{eq_16} we can rewrite the Equation~\eqref{eq_10} as
\[
\frac{\pi }{4} = {2^{k - 1}}\arctan \left( {\frac{1}{{{\beta _1}}}} \right) + \arctan\left(\frac{1-\sin\left(2^{k-1}\arctan\left(\frac{2\beta_1}{\beta_1^2-1}\right)\right)}{\cos\left(2^{k-1}\arctan\left(\frac{2\beta_1}{\beta_1^2-1}\right)\right)}\right).
\]

In general, the~constant ${\beta _1}$ may be either rational or irrational number. However, it is more convenient to apply notation ${u_{1,k}}$ that is defined by Equation~\eqref{eq_5} instead of ${\beta _1}$. Such a notation is to emphasize that the constant ${u_{1,k}}$ is an integer dependent upon on $k$. Thus, with~this notation the two-term Machin-like formula for $\pi$ can be represented as
\begin{equation}\label{eq_17}
\frac{\pi }{4} = {2^{k - 1}}\arctan \left( {\frac{1}{{{u_{1,k}}}}} \right) + \arctan \left( {\frac{1}{{{u_{2,k}}}}} \right),
\end{equation}
where in accordance with Equation~\eqref{eq_12} we have now
\begin{equation}\label{eq_18}
{u_{2,k}} = \frac{2}{{{{\left( {\left( {{u_{1,k}} + i} \right)/\left( {{u_{1,k}} - i} \right)} \right)}^{{2^{k - 1}}}} - i}} - i.
\end{equation}

It is interesting to note that by taking $k = 3$, we get
$$
{u_{1,3}} = \left\lfloor {\frac{{{a_3}}}{{\sqrt {2 - {a_2}} }}} \right\rfloor  = \left\lfloor {\frac{{\sqrt {2 + \sqrt {2 + \sqrt 2 } } }}{{\sqrt {2 - \sqrt {2 + \sqrt 2 } } }}} \right\rfloor  = 5.
$$
Substituting ${u_{1,3}} = 5$ into Equation~\eqref{eq_17} we obtain ${u_{2,3}} =  - 239.$ Considering that \linebreak ${2^{k - 1}} = {2^2} = 4$ and substituting these two constants into Equation~\eqref{eq_10} we derive an original Machin-like formula \eqref{eq_1} for $\pi$.

Since the value ${2^{k - 1}}$ rapidly increases with increasing $k$, application of the \linebreak Equation~\eqref{eq_18} if not effective to compute the second constant ${u_{2,k}}$. However, the~two-step iteration \eqref{eq_14} perfectly resolves this issue. Specifically, implying that the initial values for the two-step iteration \eqref{eq_14}  are
\[
{\sigma_1} = {\mathop{\rm Re}\nolimits} \left[ {\frac{{{u_{1,k}} + i}}{{{u_{1,k}} - i}}} \right] = \frac{{u_{_{1,k}}^2 - 1}}{{u_{_{1,k}}^2 + 1}}
\]
and
\[
{\tau_1} = {\mathop{\rm Im}\nolimits} \left[ {\frac{{{u_{1,k}} + i}}{{{u_{1,k}} - i}}} \right] = \frac{{2{u_{1,k}}}}{{u_{_{1,k}}^2 + 1}},
\]
we can find the second constant as
\begin{equation}\label{eq_19}
{u_{2,k}} = \frac{{{\sigma_k}}}{{1 - {\tau_k}}}.
\end{equation}

We can derive again the original Machin-like formula \eqref{eq_1} for $\pi$ by using the two-step iteration \eqref{eq_14} at $k = 3$. This leads to the following
\[
{\sigma_1} = \frac{{u_{_{1,3}}^2 - 1}}{{u_{_{1,3}}^2 + 1}} = \frac{{24}}{{26}}, \qquad {\tau_1} = \frac{{2{u_{1,3}}}}{{u_{_{1,3}}^2 + 1}} = \frac{{10}}{{26}},
\]
\[
{\sigma_2} = \sigma_1^2 - \tau_1^2 = \frac{{119}}{{169}}, \qquad {\tau_2} = 2{\sigma_1}{\tau_1} = \frac{{120}}{{169}},
\]
\[
{\sigma_3} = \sigma_2^2 - \tau_2^2 =  - \frac{{239}}{{28561}}, \qquad {\tau_3} = 2{\sigma_2}{\tau_2} = \frac{{28560}}{{28561}}.
\]
Finally, using Equation \eqref{eq_19} we can find the second constant to be
\[
{u_{2,3}} = \frac{{{\sigma_3}}}{{1 - {\tau_3}}} =  - \frac{{239/28561}}{{1 - 28560/28561}} =  -239.
\]

It should be noted that the second constant ${u_{2,k}}$ is an integer only at $k = 2$ and $k = 3$. At~$k>3$ it is not an integer but a rational~number.

The next example is $k = 6$. The~first constant is an integer given by
$$
{u_{1,6}} = \left\lfloor {\frac{{{a_6}}}{{\sqrt {2 - {a_5}} }}} \right\rfloor  = \left\lfloor {\frac{{\sqrt {2 + \sqrt {2 + \sqrt {2 + \sqrt {2 + \sqrt {2 + \sqrt 2 } } } } } }}{{\sqrt {2 - \sqrt {2 + \sqrt {2 + \sqrt {2 + \sqrt {2 + \sqrt 2 } } } } } }}} \right\rfloor  = 40.
$$

The second constant $u_{2,6}$ is a rational number that can be computed either by Equation~\eqref{eq_18} or, more efficiently, by~two-step iteration \eqref{eq_14}
$$
{u_{2,6}} =  - \frac{{{\rm{2634699316100146880926635665506082395762836079845121}}}}{{{\rm{38035138859000075702655846657186322249216830232319}}}}.
$$

The following Mathematica code validates the two-term Machin-like formula for $\pi$ at $k = 6$:
\vspace{-0.25cm}
\begin{shaded}
\begin{verbatim}
k=6;
\[Beta]1=40;
\[Beta]2=-2634699316100146880926635665506082395762836079845121/
38035138859000075702655846657186322249216830232319;

Pi/4==2^(k-1)*ArcTan[1/\[Beta]1]+ArcTan[1/\[Beta]2]
\end{verbatim}
\end{shaded}
\vspace{-0.25cm}
\noindent by returning {\ttfamily{True}}.

Alternatively, the~second constant can also be found by using the following identity in trigonometric form
\begin{equation}\label{eq_20}
{u_{2,k}} = \frac{{\cos \left( {{2^{k - 1}}\arctan \left( {\frac{{2{u_{1,k}}}}{{u_{1,k}^2 - 1}}} \right)} \right)}}{{1 - \sin \left( {{2^{k - 1}}\arctan \left( {\frac{{2{u_{1,k}}}}{{u_{1,k}^2 - 1}}} \right)} \right)}}
\end{equation}
that follows from Equation~\eqref{eq_16}. It should be noted that the constant ${u_{2,k}}$ must be a rational number as it has been shown by Theorem \ref{th_3}.

We can see that from Equation \eqref{eq_5} it follows that
$$
{u_{1,k}} << u_{1,k}^2, \qquad k >> 1.
$$
Consequently, the following ratio can be simplified as given by
$$
\frac{{2{u_{1,k}}}}{{u_{1,k}^2 - 1}} \approx \frac{2}{{{u_{1,k}}}}.
$$

Replacing the arguments of the sine and cosine functions in Equation~\eqref{eq_20}, we can approximate the two-term Machin-like formula \eqref{eq_17} for $\pi$ as
$$
\frac{\pi }{4} \approx {2^{k - 1}}\arctan \left( {\frac{1}{{{u_{1,k}}}}} \right) + \arctan \left( {\frac{{1 - \sin \left( {{2^k}/{u_{1,k}}} \right)}}{{\cos \left( {{2^k}/{u_{1,k}}} \right)}}} \right).
$$
Using the identities for the double angle
$$
\sin \left( x \right) = \frac{{2\tan \left( x/2 \right)}}{{1 + {{\tan }^2}\left( x/2 \right)}},
$$
$$
\cos \left( x \right) = \frac{{1 - {{\tan }^2}\left( x/2 \right)}}{{1 + {{\tan }^2}\left( x/2 \right)}},
$$
after some trivial rearrangements we obtain
\begin{equation}\label{eq_21}
\frac{\pi }{4} \approx {2^{k - 1}}\arctan \left( {\frac{1}{{{u_{1,k}}}}} \right) + \arctan \left( {\frac{{1 - \tan \left( {{2^{k - 1}}/{u_{1,k}}} \right)}}{{1 + {{\tan }^2}\left( {{2^{k - 1}}/{u_{1,k}}} \right)}}} \right).
\end{equation}

Recently, it has been noticed in publication~\cite{WolframCloud} that the ratio ${2^{k - 1}}/{u_{1,k}}$ approximates $\pi/4$ reasonably well when integer $k=1000$. The~following theorem shows why accuracy of this ratio improves with increasing $k$.

\begin{theorem}\label{th_4}
There is a limit
\begin{equation}\label{eq_22}
\mathop {\lim }\limits_{k \to \infty } \frac{{{2^{k - 1}}}}{{{u_{1,k}}}} = \frac{\pi }{4}.
\end{equation}
\end{theorem}

\begin{proof}
By definition of the floor function we have
$$
\frac{{{a_k}}}{{\sqrt {2 - {a_{k - 1}}} }} = \left\lfloor {\frac{{{a_k}}}{{\sqrt {2 - {a_{k - 1}}} }}} \right\rfloor  + {\rm{frac}}\left( {\frac{{{a_k}}}{{\sqrt {2 - {a_{k - 1}}} }}} \right),
$$
where {by definition the fractional part cannot be smaller than zero and greater than or equal to unity}
$$
0 \ge {\rm{frac}}\left( {\frac{{{a_k}}}{{\sqrt {2 - {a_{k - 1}}} }}} \right) < 1.
$$
Therefore, we can write
\[
\begin{aligned}
\mathop {\lim }\limits_{k \to \infty } \frac{{{a_k}/\sqrt {2 - {a_{k - 1}}} }}{{{u_{1,k}}}} &= \mathop {\lim }\limits_{k \to \infty } \frac{{{a_k}/\sqrt {2 - {a_{k - 1}}} }}{{\left\lfloor {{a_k}/\sqrt {2 - {a_{k - 1}}} } \right\rfloor }}\\
&= \mathop {\lim }\limits_{k \to \infty } \frac{{\left\lfloor {{a_k}/\sqrt {2 - {a_{k - 1}}} } \right\rfloor  + {\rm{frac}}\left( {{a_k}/\sqrt {2 - {a_{k - 1}}} } \right)}}{{\left\lfloor {{a_k}/\sqrt {2 - {a_{k - 1}}} } \right\rfloor }} \\
&= 1 + \mathop {\lim }\limits_{k \to \infty } \frac{{{\rm{frac}}\left( {{a_k}/\sqrt {2 - {a_{k - 1}}} } \right)}}{{\left\lfloor {{a_k}/\sqrt {2 - {a_{k - 1}}} } \right\rfloor }}=1.
\end{aligned}
\]
Since the fractional part cannot be smaller than $0$ and greater than $1$ we can conclude that
$$
\mathop {\lim }\limits_{k \to \infty } \frac{{{a_k}/\sqrt {2 - {a_{k - 1}}} }}{{{u_{1,k}}}} = 1.
$$
Consequently, we can infer that
$$
\mathop {\lim }\limits_{k \to \infty } {2^{k - 1}}\frac{{\sqrt {2 - {a_{k - 1}}} }}{{{a_k}}} = \mathop {\lim }\limits_{k \to \infty } \frac{{{2^{k - 1}}}}{{{u_{1,k}}}}.
$$

From Theorem \ref{th_2} we know that
$$
\mathop {\lim }\limits_{k \to \infty } \frac{{\sqrt {2 - {a_{k - 1}}} }}{{{a_k}}} = \frac{{\mathop {\lim }\limits_{k \to \infty } \sqrt {2 - {a_{k - 1}}} }}{{\mathop {\lim }\limits_{k \to \infty } {a_k}}} = \frac{0}{2} = 0.
$$
Consequently, from~the relation \eqref{eq_3} we get
$$
\mathop {\lim }\limits_{k \to \infty } {2^{k - 1}}\frac{{\sqrt {2 - {a_{k - 1}}} }}{{{a_k}}} = \mathop {\lim }\limits_{k \to \infty } {2^{k - 1}}\arctan \left( {\frac{{\sqrt {2 - {a_{k - 1}}} }}{{{a_k}}}} \right) = \frac{\pi }{4}
$$
and the limit \eqref{eq_22} follows.
\end{proof}

\begin{lemma}\label{lm_2}
There is a limit such that
\[
\lim_{k\to\infty}\frac{1}{u_{2,k}}=0.
\]
\end{lemma}

\begin{proof}
We know that the limit
$$
\frac{\pi}{4}=\lim_{k\to\infty}\left[2^{k-1}\arctan\left(\frac{1}{u_{1,k}}\right)+\arctan\left(\frac{1}{u_{2,k}}\right)\right]
$$
is valid since the identity \eqref{eq_17} remains valid at any arbitrarily large integer $k$. We also know that 
from the Theorem \ref{th_4} and relation \eqref{eq_3} it follows that
$$
\frac{\pi}{4}=\lim_{k\to\infty}\frac{2^{k-1}}{u_{1,k}}=\lim_{k\to\infty}2^{k-1}\arctan\left(\frac{1}{u_{1,k}}\right).
$$
This signifies that
$$
\lim_{k\to\infty}2^{k-1}\arctan\left(\frac{1}{u_{1,k}}\right)=\lim_{k\to\infty}\left[2^{k-1}\arctan\left(\frac{1}{u_{1,k}}\right)+\arctan\left(\frac{1}{u_{2,k}}\right)\right].
$$
However, according to relation \eqref{eq_3} this equation can be simplified as
$$
\lim_{k\to\infty}\frac{2^{k-1}}{u_{1,k}}=\lim_{k\to\infty}\left[\frac{2^{k-1}}{u_{1,k}}+\frac{1}{u_{2,k}}\right]
$$
and the proof follows.
\end{proof}

Applying Theorem \ref{th_4} to the left side of approximation \eqref{eq_21} yields
$$
\frac{2^{k-1}}{{{u_{1,k}}}} \approx 2^{k-1}\arctan \left( {\frac{1}{{{u_{1,k}}}}} \right) +\arctan \left( {\frac{{1 - \tan \left( {{2^{k - 1}}/{u_{1,k}}} \right)}}{{1 + {{\tan }^2}\left( {{2^{k - 1}}/{u_{1,k}}} \right)}}} \right).
$$
or
$$
{{{u_{1,k}}}} \approx \frac{1}{\arctan \left( {\frac{1}{{{u_{1,k}}}}} \right) +\frac{1}{2^{k-1}}\arctan \left( {\frac{{1 - \tan \left( {{2^{k - 1}}/{u_{1,k}}} \right)}}{{1 + {{\tan }^2}\left( {{2^{k - 1}}/{u_{1,k}}} \right)}}} \right)}.
$$

We note that both arguments of the arctangent function tend to zero with increasing $k$. Therefore, referring to the relation \eqref{eq_3} again we can simplify the approximation above as
$$
{{{u_{1,k}}}} \approx \frac{1}{{\frac{1}{{{u_{1,k}}}}} + {\frac{{1 - \tan \left( {{2^{k - 1}}/{u_{1,k}}} \right)}}{{2^{k-1}\left(1 + {{\tan }^2}\left( {{2^{k - 1}}/{u_{1,k}}} \right)\right)}}}}.
$$
Using the Theorem \ref{th_4} it follows that at $k\to\infty$
$$
{\tan ^2}\left( {\frac{{{2^{k - 1}}}}{{{u_{1,k}}}}} \right) \to {\tan ^2}\left( {\frac{\pi }{4}} \right) \to 1.
$$
Consequently, the~value
\begin{equation}\label{eq_23}
1 + {{\tan }^2}\left( \frac{2^{k - 1}}{u_{1,k}} \right) \to 2
\end{equation}
with increasing $k$. This leads to
\begin{equation}\label{eq_24}
{u_{1,k}} \approx \frac{1}{{\frac{1}{{{u_{1,k}}}} + \frac{1}{{{2^k}}}\left( {1 - \tan \left( {\frac{{{2^{k - 1}}}}{{{u_{1,k}}}}} \right)} \right)}}, \qquad k >> 1.
\end{equation}

Comparing Equations \eqref{eq_17} and \eqref{eq_21} one can see that
\[
u_{2,k} \approx \frac{{1 + \tan^2 \left( {{2^{k - 1}}/{u_{1,k}}} \right)}}{{1 - \tan \left( {{2^{k - 1}}/{u_{1,k}}} \right)}}, \qquad k >> 1
\]
and due to relation \eqref{eq_23} this approximation can be further simplified to
\begin{equation}\label{eq_25}
u_{2,k}  \approx \frac{2}{1 - \tan \left( 2^{k - 1}/u_{1,k} \right)}, \qquad k >> 1.
\end{equation}

Although this equation only approximates the second constant $u_{2,k}$, its accuracy, nevertheless, improves with increasing $k$. Perhaps, the~approximation \eqref{eq_25} can also be used at larger values of the integer $k$ as an alternative to the exact formula \eqref{eq_19} based on the two-step iteration \eqref{eq_14}.

We can see consistency of the approximations \eqref{eq_23} and \eqref{eq_25} with Theorem \ref{th_4} and Lemma \ref{lm_2}. In~particular, when $k$ tends to infinity the constants $u_{1,k}$ and $u_{2,k}$ also tend to infinity. Therefore, in~order to enhance a convergence rate, it is important to obtain the integer $u_{1,k}$ in the two-term Machin-like formula \eqref{eq_17} for $\pi$ as large as possible. Once the value of the first constant $u_{1,k}$ is determined, 
the second constant $u_{2,k}$ can be computed by using Equation~\eqref{eq_19} based on two-step iteration formula \eqref{eq_14}. For~example, at~$k = 27$ the value $u_{1,27}$ = 85,445,659. The~corresponding Lehmer's measure is $\mu\approx 0.245319$ only. Such a small Lehmer's measure implies a rapid convergence rate. In~particular, we can observe $16$ correct decimal digits of $\pi$ per term increment. This can be confirmed by running a Mathematica program provided in~\cite{Abrarov2018b}.

At $k=27$ Equation~\eqref{eq_19} yields a rational number $u_{2,27}$ consisting of 522,185,807 digits in numerator and 522,185,816 digits in denominator. Such a quotient with huge numbers in numerator and denominator is not unusual and can also be observed in Borwein integrals. Specifically, Bäsel and Baillie in their work~\cite{Basel2016} showed that a formula for $\pi$ can be generated with a quotient consisting of 453,130,145 and 453,237,170 digits in its numerator and denominator, respectively. The~interested readers can download the exact number $u_{2,27}$ with all digits from~\cite{Abrarov2017c}.

\section{Implementation}

At first glance, the approximation \eqref{eq_24} does not look interesting as its both sides contain the constant 
${u_{1,k}}$ and it is unclear how to represent it in explicit form. However, sample computations we performed with this approximation show that it can be implemented effectively. In~particular, we noticed that application of approximation \eqref{eq_24} in iteration provides a result that tends to be more accurate with increasing the integer $k$.

Consider for example $k = 10$. In~this case we have that
\[
{u_{1,10}} = \left\lfloor {\frac{{{a_{10}}}}{{\sqrt {2 - {a_9}} }}} \right\rfloor  = 651.
\]
With initial guess for ${u_{1,10}}$, say $1000$, after~just $5$ iterations (self-substitutions) we obtain
\[
{u_{1,10}} \approx 651.899.
\]
This can be seen by running the following Mathematica command lines that show the results of computation based on this iteration:
\vspace{-0.25cm}
\begin{shaded}
\begin{verbatim}
k=10;
a[1]=Sqrt[2];
a[n_]:=a[n]=Sqrt[2+a[n-1]];
Print["Exact value: ",Floor[a[10]/Sqrt[2-a[9]]]];

u2k:=1000;
itr=1;
Print["Initial guess value: ",u2k];

Print["--------------------------"];
Print["Iteration    ","Approximation"];
Print["-------------- -----------"];
While[itr<=5,u2k=1/(1/u2k+1/2^k (1-Tan[2^(k-1)/u2k]));
    Print[itr,"            ",u2k//N];itr++];
\end{verbatim}
\end{shaded}
\vspace{-0.25cm}
\noindent The Mathematica generates the following output:
\begin{alltt}
Exact value: 651
Initial guess value: 1000
--------------------------
Iteration    Approximation
--------------------------
1            700.404
2            654.196
3            651.905
4            651.899
5            651.899
\end{alltt}
We have ceased the iterative process after $5$-th cycle since two successive numbers coincide with each other at fourth and fifth iterations with same output $651.899$.

Application of the Equation~\eqref{eq_24} is convenient since for each consecutive increment of the integer $k$ the following inequality remains valid
\begin{equation}\label{eq_26}
2{u_{1,k}} \le {u_{1,k + 1}} \le 2{u_{1,k}} + 1.
\end{equation}
This inequality follows from the property of the floor function that is used in \linebreak Equation~\eqref{eq_5}; the constant ${u_{1,k + 1}}$ should be either equal to $2{u_{1,k}}$ or larger it by unity (see~\cite{WolframCloud} for some examples). Thus, based on Equation~\eqref{eq_24} and inequality \eqref{eq_26} we can make the \linebreak following assumption
\begin{equation}\label{eq_27}
{u_{1,k + 1}} = \left\lfloor {2\cdot\frac{1}{{\frac{1}{{{u_{1,k}}}} + \frac{1}{{{2^k}}}\left( {1 - \tan \left( {{2^{k - 1}}/{u_{1,k}}} \right)} \right)}}} \right\rfloor.
\end{equation}

The computational tests we performed shows that this formula provides correct results for a large range of the integer $k\ge 2$. However, its  general applicability for any arbitrarily large $k$ yet to be~proved.

There are different methods to approximate the tangent function in Equation~\eqref{eq_27}. One of the ways is to truncate the following expansion series
\begin{equation}\label{eq_28} 
\begin{aligned}
\tan \left( x \right) =& \sum\limits_{n = 1}^\infty  {\frac{{{{\left( { - 1} \right)}^{n - 1}}{2^{2n}}\left( {{2^{2n}} - 1} \right){B_{2n}}}}{{\left( {2n} \right)!}}{x^{2n - 1}}} \\
=& x + \frac{{{x^3}}}{3} + \frac{{2{x^5}}}{{15}} + \frac{{17{x^7}}}{{315}} + \frac{{62{x^9}}}{{2835}} +  \cdots  \Leftrightarrow \tan \left( x \right) = x + O\left( {{x^3}} \right),
\end{aligned}
\end{equation}
where ${B_n}$ are the Bernoulli numbers, defined by a contour integral
$$
{B_n} = \frac{{n!}}{{2\pi i}}\oint {\frac{z}{{{e^z} - 1}}\frac{{dz}}{{{z^{n + 1}}}}}.
$$
Although this series expansion is rapid in convergence, its application may not be optimal since it requires the determination of the Bernoulli numbers. One of the ways to compute them is given by the following identity
$$
{B_n} = \sum\limits_{m = 0}^n {\frac{1}{{m + 1}}} \sum\limits_{\ell  = 0}^m {{{\left( { - 1} \right)}^\ell }\left( \begin{aligned}
m\\
\ell
\end{aligned} \right){\ell ^n}}.
$$
We can see that this formula involves the double summation and, therefore, cannot be rapid in principle especially at larger orders of $n$. Although~other methods of computation of the Bernoulli numbers are more efficient, their implementations require quite sophisticated algorithms~\cite{Knuth1967, Koepf1992, Harvey2010}.

Alternatively, the~tangent function may also be computed by using continued fractions~\cite{Trott2011, Havil2012, Oliver2012}. However, algorithmic implementation of the continued fractions may not be optimal for our particular~task.

This problem can be resolved by noticing that at each consecutive step of iteration the integer ${u_{1,k}}$ increases, and because of this the argument of the tangent function decreases. Consequently, it may be reasonable to utilize argument reduction method for computation of the tangent function~\cite{Beebe2017}. As~a simplest case we can use, for~example, the double angle identity providing argument reduction by a factor of two
$$
\tan \left(2 x \right) = \frac{{2\tan \left( {x} \right)}}{{1 - {{\tan }^2}\left( {x} \right)}}.
$$
Therefore, taking into account that in accordance with \eqref{eq_28} $\tan\left(x\right)\to x$ at $x \to 0$, we can approximate the double angle identity above as
\[
\tan\left(2x\right) \approx \frac{2x}{{1 - {{{x}}^2}}}, \qquad \left| x \right| << 1.
\]
This approximation implies that the arctangent function can be calculated in a simple iteration over and over again by defining the following function
\begin{equation}\label{eq_29}
f_n\left(x\right) = \frac{2f_{n-1}\left(x\right)}{1 - f_{n-1}^2\left(x\right)}\approx \tan\left(2^{n}x\right),
\end{equation}
where
\[
f_1\left(x\right) = \frac{2x}{1 - x^2}.
\]
Tangent 
 function can be computed more accurately by defining $$f_1\left(x\right) = \frac{2\left(x+x^3/3\right)}{1 - \left(x+x^3/3\right)^2}$$ since according to expansion series \eqref{eq_28} we can also infer that $\tan\left(x\right)=x+x^3/3+O\left(x^5\right)$.

The following Mathematica command lines execute the program for computation of the integer constant ${u_{1,k}}$ by using Equations \eqref{eq_27} and \eqref{eq_29}:
\vspace{-0.25cm}
\small
\begin{shaded}
\begin{verbatim}
(* Clear previous value *)
Clear[\[Beta]1];
(* Set of nested radicals *)
a[0]=0;a[k_]:=a[k]=Sqrt[2+a[k-1]];

(* Equation (5) *)
\[Beta]1[k_]:=\[Beta]1[k]=Floor[a[k]/Sqrt[2-a[k-1]]];

(* Applying Equation (29) in iteration *)
f[x_,1]:=f[x,1]=SetPrecision[(2*x)/(1-x^2),k];
f[x_,n_]:=f[x,n]=(2*f[x,n-1])/(1-f[x,n-1]^2);

(*Main computation*)
func[u1_,k_]:=func[u1,k]=1/(1/u1+1/2^k*(1-f[1/u1,k-1]));

k=2; (* iteger k *)
fstConst=2; (* first constant *)
kMax=30; (* max number for iteration *)
str={{"Integer k","  |  ","Equation (5)","  |  ","Equation (27)"}};

While[k<=kMax,AppendTo[str,{k,"  |  ",\[Beta]1[k],"  |  ",fstConst}];
    fstConst=Floor[func[2*fstConst,k+1]];k++];
Print[TableForm[str]];
\end{verbatim}
\end{shaded}
\vspace{-0.25cm}
\normalsize
\noindent The Mathematica generates the following table:
\begin{alltt}
Integer k  |  Equation (5)  |  Equation (27)
2          |  2             |  2
3          |  5             |  5
4          |  10            |  10
5          |  20            |  20
6          |  40            |  40
7          |  81            |  81
8          |  162           |  162
9          |  325           |  325
10         |  651           |  651
11         |  1303          |  1303
12         |  2607          |  2607
13         |  5215          |  5215
14         |  10430         |  10430
15         |  20860         |  20860
16         |  41721         |  41721
17         |  83443         |  83443
18         |  166886        |  166886
19         |  333772        |  333772
20         |  667544        |  667544
21         |  1335088       |  1335088
22         |  2670176       |  2670176
23         |  5340353       |  5340353
24         |  10680707      |  10680707
25         |  21361414      |  21361414
26         |  42722829      |  42722829
27         |  85445659      |  85445659
28         |  170891318     |  170891318
29         |  341782637     |  341782637
30         |  683565275     |  683565275
\end{alltt}

Thus, we can see the feasibility of computation of the constant ${u_{1,k}}$ without any irrational (surd) numbers. Just by applying only arithmetic manipulations (summations, multiplications and divisions) we can compute the integer ${u_{1,k}}$ by iteration based on Equation~\eqref{eq_27}.

\section{Quadratic convergence}

There is another interesting application of Equation~\eqref{eq_24}. In~particular, we found experimentally that the following formula defined by iteration
\begin{equation}\label{eq_30}
\theta_{n+1} = \frac{1}{{\frac{1}{\theta_n} + \frac{1}{{{2^k}}}\left( {1 - \tan \left( {\frac{{{2^{k - 1}}}}{{\theta_n}}} \right)} \right)}},
\end{equation}
leads to a quadratic convergence to the constant $\pi$ such that (by assumption)
\[
\pi = \lim_{n\to\infty}\frac{2^{k+1}}{\theta_n}, \qquad k\ge 1.
\]

The quadratic convergence to $\pi$ can be observed by running the command lines:
\vspace{-0.25cm}
\small
\begin{shaded}
\begin{verbatim}
Clear[k,\[Theta]]

k=7;(* assign value of k *)
\[Theta]=2^k;(* initial guess *)
str={{"Iteration No."," | ","Computed digits of \[Pi]"}};

If[k>17,Print["Please wait. Computing..."]];

(* Equation (30) used in iteration *)
For[n=1,If[k<15,n<=15,n<=k],n++,\[Theta]=SetPrecision[1/(1/\[Theta]+
    1/2^k*(1-Tan[2^(k-1)/\[Theta]])),2^(n+1)];AppendTo[str,
        {n," | ",MantissaExponent[Pi-2^(k+1)/\[Theta]][[2]]//Abs}]];

Print[str//TableForm];
\end{verbatim}
\end{shaded}
\vspace{-0.25cm}
\normalsize
\noindent The output of Mathematica is the following table:
\begin{alltt}
Iteration No. | Computed digits of \(\pi\)
1             | 0
2             | 1
3             | 4
4             | 9
5             | 19
6             | 39
7             | 79
8             | 159
9             | 319
10            | 639
11            | 1278
12            | 2558
13            | 5116
14            | 10233
15            | 20468
\end{alltt}
As we can see from this table, after~third iteration the number of correct digits of $\pi$ increases by factor of two at each consecutive step of~iteration.

{More explicitly, the~dynamics of computation of $\pi$ can be seen by running the following Mathematica code:}
\vspace{-0.25cm}
\begin{shaded}
\begin{verbatim}
Clear[k,\[Theta]]

k=7;(*assign value of k*)
\[Theta][0]:=2^k;(*initial guess*)

(* Iteration formula (30) *)
\[Theta][n_]:=1/(1/\[Theta][n-1]+1/2^k*(1-Tan[2^(k-1)/
    \[Theta][n-1]]));

(* Approximated value of \[Pi] *)
piAppr[n_]:=2^(k+1)/\[Theta][n];

Print["Iteration 1"];
Print[N[piAppr[1],25],"..."];

Print["Iteration 2"];
Print[N[piAppr[2],25],"..."];

Print["Iteration 3"];
Print[N[piAppr[3],25],"..."];

Print["Iteration 4"];
Print[N[piAppr[4],25],"..."];

Print["Iteration 5"];
Print[N[piAppr[5],25],"..."];

Print["------------------"];
Print["Actual value of \[Pi]"];
Print[N[Pi,25],"..."];
\end{verbatim}
\end{shaded}
\vspace{-0.25cm}
\noindent Mathematica returns the following output:
\begin{alltt}
Iteration 1
2.907395020312418973489641...

Iteration 2
3.128878092399718501843067...

Iteration 3
3.141552409181815125317050...

Iteration 4
3.141592653184895576712223...

Iteration 5
3.141592653589793238421658...

Actual value of \(\pi\)
3.141592653589793238462643...
\end{alltt}
{As we can see, the~first five iterations provide $0$, $1$, $4$, $9$ and $19$ correct decimal digits of $\pi$, respectively. The~actual value of $\pi$, generated by Mathematica built-in function, is also shown for comparison.}

The quadratic convergence to $\pi$ can be implemented by using the Brent--Salamin algorithm~\cite{Brent1976, Salamin1976,Newman1985,Lord1992} (It is also 
 known as the Gauss--Brent--Salamin algorithm). However, in~contrast to the Brent--Salamin algorithm the proposed iteration formula \eqref{eq_30} provides quadratic convergence to $\pi$ without any irrational (surd) numbers.

\section{Conclusions}

In this work we propose a method for determination of the integer $u_{1,k}$. In~particular, the~algorithmic implementation of the Formula \eqref{eq_27} shows that it can be used as an alternative to Equation~\eqref{eq_5} requiring a set of the nested radicals $\left\{{a_k}\right\}$ defined as $a_{k}=\sqrt{2+a_{k-1}}$ and $a_1=0$. This method is based on a simple iteration and can be implemented without any irrational (surd) numbers.

\section*{Acknowledgments}
This work is supported by National Research Council Canada, Thoth Technology Inc., York University and Epic College of~Technology.

\end{document}